\documentclass[a4paper, 11pt, reqno]{amsart}
\usepackage[english]{babel}
\usepackage[utf8]{inputenc}
\usepackage[T1]{fontenc}
\usepackage{amsthm, amsmath, amssymb, amsrefs, enumerate, enumitem, mathtools, lmodern, microtype, mathrsfs, listings, tikz-cd, comment, float}
\usepackage{a4wide}
\usepackage{hyperref}
\usepackage{nicefrac}
\usepackage{tikz}
\usetikzlibrary{shapes,positioning,intersections,quotes}

\theoremstyle{plain}
\newtheorem{thm}{Theorem}[section]
\newtheorem{prop}[thm]{Proposition}
\newtheorem{lemma}[thm]{Lemma}
\newtheorem{cor}[thm]{Corollary}

\theoremstyle{definition}

\theoremstyle{remark}
\newtheorem*{rmk}{Remark}

\makeatletter
\let\@@pmod\pmod
\DeclareRobustCommand{\pmod}{\@ifstar\@pmods\@@pmod}
\def\@pmods#1{\mkern4mu({\operator@font mod}\mkern 6mu#1)}
\makeatother

\binoppenalty=\maxdimen
\relpenalty=\maxdimen

\numberwithin{equation}{section}
\setlist{nosep}
\setlist{noitemsep}

\lstdefinelanguage{Sage}[]{Python}
{morekeywords={False,sage,True},sensitive=true}

\newcommand{\C}{\mathbb{C}}
\renewcommand{\H}{\mathbb{H}}
\newcommand{\Z}{\mathbb{Z}}
\newcommand{\Q}{\mathbb{Q}}
\newcommand{\N}{\mathbb{N}}
\newcommand{\R}{\mathbb{R}}

\newcommand{\slz}{{\text {\rm SL}}_2(\mathbb{Z})}

\DeclarePairedDelimiter\floor{\lfloor}{\rfloor}

\newcommand{\F}{\mathbb{F}}

\title{On the Detection of Non-Roots of D'Arcais Polynomials}
\author{Bernhard Heim}
\author{Johann Stumpenhusen}
\address{Department of Mathematics and Computer Science, Division of Mathematics, University of Cologne, Weyertal 86--90, 50931 Cologne, Germany}
\email{bheim@uni-koeln.de}
\email{jstumpen@math.uni-koeln.de}
\date{\today}

\begin{document}

\maketitle

\begin{abstract}
    The Lehmer conjecture states that the non-constant Fourier coefficients of the 24th power of the Dedekind eta function are non-zero. In a recent preprint, Neuhauser and the first author exploited an easily accessible tool from algebraic number theory, namely the Dedekind--Kummer Theorem, to prove the non-vanishing of the Fourier coefficients of certain powers of the Dedekind eta function at roots of unity. We extend the application of this method to enlarge the scope of non-roots of the related D'Arcais polynomials.
\end{abstract}

\section{Introduction and statement of results}

\subsection{D'Arcais polynomials.} Let
\[\eta(\tau) := q^{\frac{1}{24}}\prod_{n = 1}^\infty (1 - q^n), \qquad q := e^{2\pi i \tau}, \tau \in \H,\]
be the Dedekind eta function which appears as a critical object in combinatorics and number theory as it is related to the reciprocal of the generating function of the partition function. Since $\eta(\tau)$ is a modular form of weight $\frac{1}{2}$ for $\slz$, it admits a Fourier expansion. The vanishing properties of the Fourier coefficients of the powers $\eta(\tau)^{-X}$ where $X$ is usually assumed to be a rational integer\footnote{By the term \textit{rational} integer, we refer to $\Z$ in particular as opposed to a general algebraic integer.} have been of great interest, dating back to Euler and Jacobi investigating pentagonal and triangular numbers  (see Koehler \cite{Koehler}, Ono \cite{On03}). In fact, we can rewrite the infinite product as
\begin{equation*}\label{eq:DefD'ArcaisPolynomials}
    \prod_{n = 1}^\infty (1 - q^n)^{-X} = \exp \left(X \sum_{n = 1}^\infty \sigma(n) \frac{q^n}{n}\right) =: \sum_{n = 0}^\infty P_n^\sigma(X)q^n
\end{equation*}
where $\sigma(n) := \sum_{d \mid n} d$ is the usual sum-of-divisors-function. The polynomials $P_n^\sigma(X)$ are called \textit{D'Arcais polynomials} (or \textit{Nekrasov}--\textit{Okounkov} polynomials in combinatorics) and the main interest of this paper as $P_n^\sigma(X)$ equals the $(n+\frac{X}{24})$-th Fourier coefficient of $\eta(\tau)^{-X}$. The Lehmer conjecture \cite{Lehmer} states that $P_n^\sigma(X)$ never vanishes for $X = -24$.

Whereas determining roots of the D'Arcais polynomials remains a more than challenging task, there has been some progress on detecting non-roots. Kostant \cite{Kostant} proved that $P_n^\sigma(-m^2+1)$ does not vanish for $m \geq n$ using the theory of simple Lie algebras. Recently, Neuhauser and the first author \cite{HeimNeu24} improved a result by Han \cite{Han10} from 2010: If $c := 9.7226$ and $z \in \C$ such that $|z| > c(n-1)$, then $P_n^\sigma(z) \neq 0$.

The superscript $\sigma$ already teased a generalization. Replacing $\sigma$ by a $\Z$-valued arithmetic function $g: \N \longrightarrow \Z$ with $g(1) = 1$, we instead obtain
\[\sum_{n = 0}^\infty P_n^g(X)q^n := \prod_{n = 1}^\infty (1 - q^n)^{-f_g(n)X} = \exp \left(X \sum_{n = 1}^\infty g(n) \frac{q^n}{n}\right)\]
where $f_g(n) := \frac{1}{n} \sum_{d \mid n}\mu(d)g\left(\frac{n}{d}\right)$ is a rescaled convolution of $g$ with the usual M\"obius function $\mu$. Notice that, by M\"obius inversion, we have $f_\sigma(n) = 1$ all $n$. Choosing $g(n)$ to be the number of points on a smooth projective variety over the unique extension of degree $n$ over a finite field, this construction is connected to the Weil conjectures (see \cite{Murty}*{Subchapter 2.1}).

\subsection{Previous work} Let $\zeta_m$ be a primitive $m$-th root of unity and $\Phi_m(X)$ be its minimal polynomial, the $m$-th cyclotomic polynomial. Luca, Neuhauser, and the first author used local obstructions modulo a prime number $p$ to detect non-roots of $P_n^\sigma(X)$ at $\zeta_m$ for $m \geq 3$.

\begin{thm}[Heim--Luca--Neuhauser \cite{HeimLucaNeu}*{Theorem 1}]
    There is no pair of positive rational integers $(m,n)$ with $m \geq 3$ such that $\Phi_m(X) \mid P_n^\sigma(X)$.
\end{thm}

Shortly thereafter, \.{Z}mija discovered that this result is not tied to the case $g = \sigma$ but rather to the local behavior of $P_n^g(X)$ and generalized it.

\begin{thm}[\.{Z}mija \cite{Zmija}*{Theorem 1}]\label{thm:Zmija}
    Let $g$ be an arithmetic function such that $g(1) = 1$ and for every $n$, the polynomial $P_n^g(X)$ takes rational integer values at rational integer arguments. Assume moreover that if $A_n^g(X) := n!P_n^g(X)$, then
    \begin{enumerate}
        \item modulo $5$: none of the polynomials $A_3^g(X)$ and $A_4^g(X)$ is divisible by a monic irreducible polynomial of degree $2$ over $\F_5$.
        \item modulo $7$: none of the polynomials $A_r^g(X)$ for $2 \leq r \leq 6$ is divisible by a monic irreducible polynomial of degree $4$ over $\F_7$.
        \item modulo $11$: none of the polynomials $A_r^g(X)$ for $2 \leq r \leq 10$ is divisible by a monic irreducible polynomial over $\F_{11}$ that divides $X^{11^6 - 1} - 1$ and does not divide $X^{11^d- 1}- 1$ for $1 \leq d \leq 10, d \neq 6$.
    \end{enumerate}
    Then $P_n^g(\zeta_m) \neq 0$ for all $m \geq 3$.
\end{thm}

In a recent preprint, Neuhauser and the first author reduced the amount of assumptions to a simple condition on the function $g$ modulo $3$ and extended the range of non-roots of the D'Arcais polynomials. Surprisingly, their proof relies on invoking a well-known and fundamental theorem from algebraic number theory: the Dedekind--Kummer Theorem.

\begin{thm}[Heim--Neuhauser \cite{HeimNeu25}*{Theorem 4}]\label{thm:NoRootsAtAlgebraicIntegersThatAreCongruentToPrimitiveRootsOfUnity}
    Let $g: \N \longrightarrow \Z$ be an arithmetic function with $g(1) = 1$ and $m \geq 3$.
    \begin{enumerate}
        \item If there exists a prime $p \neq 2$ such that $p \mid m$, then
        \[P_n^g(\zeta_m + 2\alpha) \neq 0\]
        for all $\alpha \in \Z[\zeta_m]$.
        \item If there exists a prime $p > 3$ such that $p \mid m$ or $4 \mid m$ and if $g(3) \equiv 0,1 \mod 3$, then
        \[P_n^g(\zeta_m + 3\alpha) \neq 0\]
        for all $\alpha \in \Z[\zeta_m]$.
    \end{enumerate}
\end{thm}

\subsection{Results} Our main results partially tackle the challenges posed by Neuhauser and the first author \cite{HeimNeu25} as to widen the range of known non-roots of the D'Arcais polynomials. The way the Dedekind--Kummer Theorem is applied in the proofs of Theorem \ref{thm:NoRootsAtAlgebraicIntegersThatAreCongruentToPrimitiveRootsOfUnity} heavily relies on control over the ring $\Z[\alpha]$ where $\alpha$ is an algebraic integer. We extend the scope of how to work around the obstructions that depend on such an $\alpha$ and provide some more non-roots of the D'Arcais polynomials. As our first result, we improve a result by Neuhauser and the first author regarding the Gaussian integers $\Z[i]$ and $P_n^\sigma(X)$.

\begin{thm}\label{thm:NonRootsOnImaginaryAxisWithSigma}
    Let $a, b \in \Z$ and $n \in \N$.
    \begin{enumerate}
        \item If $n \not \equiv 5 \mod 7$ and $21 \nmid a$, then $P_n^\sigma(ai + b) \neq 0$.
        \item If $n \equiv 5 \mod 7$ and either
        \begin{equation*}
            \textit{(i)} \; 3 \nmid a \qquad \textup{or} \qquad
            \textit{(ii)} \; a \not \equiv 0, \pm 1 \mod 7 \qquad \textup{or} \qquad
            \textit{(iii)} \; 7 \nmid a \textup{ and } 7 \nmid b,
        \end{equation*}
        then $P_n^\sigma(ai + b) \neq 0$.
    \end{enumerate}
\end{thm}

\begin{rmk}
    In the figure below, we display known non-roots of the $P_n^\sigma(X)$ in $\Q(i)$ near the origin. Note that $P_n^\sigma(3k \pm 1) \neq 0$ for $k \in \Z$ as well as $P_n^\sigma(\pm 3i) \neq 0$ was already proven by Neuhauser and the first author. As of now, it appears that each improvement of the modulus $m$ is tied to explicitly evaluating $A_n^\sigma(X)$ for all $n \leq m$.
\end{rmk}

\begin{figure}[H]
    \centering
    \begin{tikzpicture}
    \node[below right=2pt of {(0,0)},fill=white]{$0$};
    \node[below right=2pt of {(0,4)}]{$i\mathbb{Z}$};
    \node[below left=2pt of {(6,0)}]{$\mathbb{Z}$};
    
    \draw[gray, ultra thin] (-6, 0) -- (0, 0);
    \draw[gray, ultra thin] (0, -4) -- (0, 4);
    
    \draw[gray, ultra thin] (-6, 1) -- (6, 1);
    \draw[gray, ultra thin] (-6, 2) -- (6, 2);
    \draw[gray, ultra thin] (-6, 3) -- (6, 3);
    \draw[gray, ultra thin] (-6, -1) -- (6, -1);
    \draw[gray, ultra thin] (-6, -2) -- (6, -2);
    \draw[gray, ultra thin] (-6, -3) -- (6, -3);
    
    \draw[gray, ultra thin] (5, 4) -- (5, -4);
    \draw[gray, ultra thin] (4, 4) -- (4, -4);
    \draw[gray, ultra thin] (3, 4) -- (3, -4);
    \draw[gray, ultra thin] (2, 4) -- (2, -4);
    \draw[gray, ultra thin] (1, 4) -- (1, -4);
    \draw[gray, ultra thin] (-1, 4) -- (-1, -4);
    \draw[gray, ultra thin] (-2, 4) -- (-2, -4);
    \draw[gray, ultra thin] (-3, 4) -- (-3, -4);
    \draw[gray, ultra thin] (-4, 4) -- (-4, -4);
    \draw[gray, ultra thin] (-5, 4) -- (-5, -4);

        \draw[black, ultra thick] (0, 0) -- (6,0);

        \fill[black] (0,1) circle [radius=3pt];
        \fill[black] (0,2) circle [radius=3pt];
        \fill[black] (0,3) circle [radius=3pt];
        \fill[black] (0,-1) circle [radius=3pt];
        \fill[black] (0,-2) circle [radius=3pt];
        \fill[black] (0,-3) circle [radius=3pt];
        \filldraw[black,fill=white] (1,1) circle [radius=3pt];
        \filldraw[black,fill=white] (1,2) circle [radius=3pt];
        \filldraw[black,fill=white] (1,3) circle [radius=3pt];
        \filldraw[black,fill=white] (1,-1) circle [radius=3pt];
        \filldraw[black,fill=white] (1,-2) circle [radius=3pt];
        \filldraw[black,fill=white] (1,-3) circle [radius=3pt];
        \filldraw[black,fill=white] (2,1) circle [radius=3pt];
        \filldraw[black,fill=white] (2,2) circle [radius=3pt];
        \filldraw[black,fill=white] (2,3) circle [radius=3pt];
        \filldraw[black,fill=white] (2,-1) circle [radius=3pt];
        \filldraw[black,fill=white] (2,-2) circle [radius=3pt];
        \filldraw[black,fill=white] (2,-3) circle [radius=3pt];
        \filldraw[black,fill=white] (3,1) circle [radius=3pt];
        \filldraw[black,fill=white] (3,2) circle [radius=3pt];
        \filldraw[black,fill=white] (3,3) circle [radius=3pt];
        \filldraw[black,fill=white] (3,-1) circle [radius=3pt];
        \filldraw[black,fill=white] (3,-2) circle [radius=3pt];
        \filldraw[black,fill=white] (3,-3) circle [radius=3pt];
        \filldraw[black,fill=white] (4,1) circle [radius=3pt];
        \filldraw[black,fill=white] (4,2) circle [radius=3pt];
        \filldraw[black,fill=white] (4,3) circle [radius=3pt];
        \filldraw[black,fill=white] (4,-1) circle [radius=3pt];
        \filldraw[black,fill=white] (4,-2) circle [radius=3pt];
        \filldraw[black,fill=white] (4,-3) circle [radius=3pt];
        \filldraw[black,fill=white] (5,1) circle [radius=3pt];
        \filldraw[black,fill=white] (5,2) circle [radius=3pt];
        \filldraw[black,fill=white] (5,3) circle [radius=3pt];
        \filldraw[black,fill=white] (5,-1) circle [radius=3pt];
        \filldraw[black,fill=white] (5,-2) circle [radius=3pt];
        \filldraw[black,fill=white] (5,-3) circle [radius=3pt];
        \filldraw[black,fill=white] (-1,1) circle [radius=3pt];
        \filldraw[black,fill=white] (-1,2) circle [radius=3pt];
        \filldraw[black,fill=white] (-1,3) circle [radius=3pt];
        \filldraw[black,fill=white] (-1,-1) circle [radius=3pt];
        \filldraw[black,fill=white] (-1,-2) circle [radius=3pt];
        \filldraw[black,fill=white] (-1,-3) circle [radius=3pt];
        \filldraw[black,fill=white] (-2,1) circle [radius=3pt];
        \filldraw[black,fill=white] (-2,2) circle [radius=3pt];
        \filldraw[black,fill=white] (-2,3) circle [radius=3pt];
        \filldraw[black,fill=white] (-2,-1) circle [radius=3pt];
        \filldraw[black,fill=white] (-2,-2) circle [radius=3pt];
        \filldraw[black,fill=white] (-2,-3) circle [radius=3pt];
        \filldraw[black,fill=white] (-3,1) circle [radius=3pt];
        \filldraw[black,fill=white] (-3,2) circle [radius=3pt];
        \filldraw[black,fill=white] (-3,3) circle [radius=3pt];
        \filldraw[black,fill=white] (-3,-1) circle [radius=3pt];
        \filldraw[black,fill=white] (-3,-2) circle [radius=3pt];
        \filldraw[black,fill=white] (-3,-3) circle [radius=3pt];
        \filldraw[black,fill=white] (-4,1) circle [radius=3pt];
        \filldraw[black,fill=white] (-4,2) circle [radius=3pt];
        \filldraw[black,fill=white] (-4,3) circle [radius=3pt];
        \filldraw[black,fill=white] (-4,-1) circle [radius=3pt];
        \filldraw[black,fill=white] (-4,-2) circle [radius=3pt];
        \filldraw[black,fill=white] (-4,-3) circle [radius=3pt];
        \filldraw[black,fill=white] (-5,1) circle [radius=3pt];
        \filldraw[black,fill=white] (-5,2) circle [radius=3pt];
        \filldraw[black,fill=white] (-5,3) circle [radius=3pt];
        \filldraw[black,fill=white] (-5,-1) circle [radius=3pt];
        \filldraw[black,fill=white] (-5,-2) circle [radius=3pt];
        \filldraw[black,fill=white] (-5,-3) circle [radius=3pt];
    \end{tikzpicture}
    \label{Gausspic}
    \caption{The Gaussian integers near the origin, each represented by an intersection of two lines. Previously known non-roots of each $P_n^\sigma(X)$ in black, newly found marked with empty points.}
\end{figure}

Additionally, it is not only possible to close some gaps for the standard case $\sigma$ but also for a more general $g$ in two of the most standard cases of number fields: cyclotomic or quadratic fields. For the latter case, set
\[\omega_D := \begin{cases}
    \sqrt{D}, &D \not \equiv 1 \mod 4,\\
    \frac{1 + \sqrt{D}}{2}, &D \equiv 1 \mod 4,
\end{cases}\]
for square-free $D \in \Z \setminus \{0,1\}$. 

\begin{thm}\label{thm:NonRootsTranslatingWithRationalIntegers}
    Let $a, b \in \Z$, $n \in \N$ and $g: \N \longrightarrow \Z$ be an arithmetic function such that $g(1) = 1$ and $m \geq 3$.
        \begin{enumerate}
            \item If $m$ is divisible by an odd rational prime and $2 \nmid a$, then
            $P_n^g(a\zeta_m + b) \neq 0.$
            \item If $m$ is divisble by a rational prime $p > 3$ or by $4$ and $3 \nmid a$ as well as $g(3) \equiv 0,1 \mod 3$, then
            $P_n^g(a\zeta_m + b) \neq 0.$
            \item Let $D \equiv 5 \mod 8$ be square-free. If $2 \nmid a$, then $P_n^g(a\omega_D + b) \neq 0.$
            \item Let $D \equiv 2 \mod 3$ be square-free. If $3 \nmid a$ and $g(3) \equiv 0, 1 \mod 3$, then
            $P_n^g(a\omega_D + b) \neq 0.$
        \end{enumerate}
\end{thm}

The local behavior of $g$ appears to be far more influential than one might guess at first glance. We employ one of our previous intermediate results to derive this last theorem.

\begin{thm}\label{thm:NotRamifiedThm}
    Let $a, b \in \Z$ and $g: \N \longrightarrow \Z$ be an arithmetic function such that $g(1) = 1$. Assume further that $D \in \Z \setminus \{0,1\}$ is square-free. If there exists a rational prime $p$ such that either
    \begin{enumerate}
        \item $g(p) \equiv 0 \mod p$ and $p \nmid 2aD$ or
        \item $g(p) \equiv 1 \mod p$ and $\left(\frac{D}{p}\right) = -1,$
    \end{enumerate} then we have $P_n^g(a\omega_D + b) \neq 0$ for all $n \equiv 0, 1 \mod p$.
\end{thm}

\section{Preliminaries}

We collect some basic ingredients that are essential for proving our results. For more details we refer to Neukirch \cite{Neukirch} for the first subsection and to \.{Z}mija \cite{Zmija} and Heim--Neuhauser \cite{HeimNeu25} for the second one.

\subsection{Algebraic number theory and the Dedekind--Kummer Theorem}

As usual, if $K$ is a number field, we denote its ring of integers by $\mathcal{O}_K$. The ring $\mathcal{O}_K$ is a Dedekind domain and as such every non-zero ideal $\mathfrak{a} \subset \mathcal{O}_K$ admits a decomposition
\begin{equation*}
    \mathfrak{a} = \mathfrak{p}_1^{e_1}\ldots \mathfrak{p}_g^{e_g}
\end{equation*}
where $\mathfrak{p}_1, \ldots, \mathfrak{p}_g \subset \mathcal{O}_K$ are prime ideals and $g, e_1, \ldots, e_g \in \N$. Additionally, for each prime ideal $\mathfrak{p}_j$, the quantity $f_j := \left|\nicefrac{\mathcal{O}_K}{\mathfrak{p}_j}\right|$ is finite. The $e_j$s are called ramification indices whereas the $f_j$s are called inertia indices.

If $\alpha$ is an algebraic integer, the index of $\alpha$
\[\kappa_\alpha := \left[\mathcal{O}_{\Q(\alpha)} : \Z[\alpha]\right]\]
is finite.

\begin{thm}[Dedekind--Kummer]\label{thm:DedekindKummer}
    Let $\alpha$ be an algebraic integer and $m_\alpha(X)$ its minimal polynomial over $\Q$. For any prime $p$ not dividing $\kappa_\alpha$, let
    \[m_\alpha(X) \equiv m_{\alpha,1}(X)^{e_1} \ldots m_{\alpha,g}(X)^{e_g} \mod p\]
    be a decomposition into pairwise distinct polynomials $m_{\alpha,1}(X), \ldots, m_{\alpha,g}(X)$ that are irreducible modulo $p$. Then we have
    \[p\mathcal{O}_{\Q(\alpha)} = \mathfrak{p}_1^{e_1} \ldots \mathfrak{p}_g^{e_g}\]
    as a decomposition into prime ideals in $\mathcal{O}_{\Q(\alpha)}$ where
    \[\mathfrak{p}_j = p\mathcal{O}_{\Q(\alpha)} + m_{\alpha,j}(\alpha)\mathcal{O}_{\Q(\alpha)}.\]
    In particular, $f_j = \deg m_{\alpha,j}(X)$.
\end{thm}

The Dedekind--Kummer Theorem provides the base for an exact procedure to determine all the prime ideals of an algebraic number field due to the Primitive Element Theorem.

A prime $p$ in a number field $K$ is said to ramify in $L \supset K$ if in its decomposition into prime ideals of $\mathcal{O}_L$, there exists an index $j$ such that $e_j > 1$. It is known that $p$ only ramifies if and only if it divides the discriminant of the extension $L/K$. In the special case of $\Q(\alpha)$ being Galois over $\Q$, the indices are equal in the sense that $e = e_1 = \ldots = e_g$ and $f = f_1 = \ldots = f_g$.

In the case of cyclotomic fields, the inertia index of a rational prime $p$ may be easily computed. Let $m$ be a natural number and $\zeta_m$ an $m$-th primitive root of unity. Then $\mathcal{O}_{\Q(\zeta_m)} = \Z[\zeta_m]$ and, since cyclotomic fields are Galois extensions over the rationals, we have that the unique inertia index of a prime $p$ in $\Q(\zeta_m)$ is the least natural number $f$ such that $p^f \equiv 1 \mod m_p$ where $m = p^{v_p(m)}m_p$ with $p \nmid m_p$. Furthermore, $p$ ramifies if and only if $p \mid m$.

\subsection{The index of a suborder} 

A critical aspect of the Dedekind--Kummer Theorem is that the prime $p$ we use as a ``lense'' is valid as long as it does not divide the index $\kappa_\alpha$. Usually, the power basis generated by an algebraic integer $\alpha$ only satisfies the requirement of being a $\Q$-basis of $\Q(\alpha)$ but is not an integral basis for $\mathcal{O}_{\Q(\alpha)}$. In other words, we often have $\kappa_\alpha > 1$. A prominent example was presented by Dedekind who investigated $\Q(\theta)$ where $\theta$ is a root of the polynomial $g(X) = X^3 - X^2 - 2X - 8$. Given an integral basis of a subring $\mathcal{O}$ of $\mathcal{O}_K$ for some number field $K$, the following well-known lemma equips us with a handy method to compute the index $[\mathcal{O}_K : \mathcal{O}]$.

\begin{lemma}[\cite{SteTall}*{Theorem 1.17}] \label{lem:IndexIsDetOfMatrix}
    Let $A$ be a free $\Z$-module of rank $k \in \N$ and $B \subset A$ a submodule of rank $k$ as well. Let their $\Z$-bases be given by $\{a_1, \ldots, a_k\}$ and $\{b_1, \ldots, b_k\}$, respectively. Then there exists a matrix $M \in \Z^{k \times k}$ such that
    \[\begin{pmatrix}
        b_1\\
        \vdots\\
        b_k
    \end{pmatrix} = M \begin{pmatrix}
        a_1\\
        \vdots\\
        a_k
    \end{pmatrix}\]
    and it holds
    \[|\det(M)| = [A : B].\]
\end{lemma}

\subsection{The splitting of $A_n^g(X)$ modulo a prime $p$}

In the first subsection of this section, the dependency between the splitting of $m_\alpha(X)$ modulo $p$ and the ramification and inertia indices of $p$ in $\Q(\alpha)$ was established. The equivalence
\[P_n^g(\alpha) = 0 \qquad \Longleftrightarrow \qquad m_\alpha(X) \mid P_n^g(X)\footnote{This division property is to be understood over $\Q[X]$.}\]
motivates the investigation of $A_n^g(X) = n!P_n^g(X)$ modulo $p$ since the $P_n^g(X)$ are---as opposed to the $A_n^g(X)$---not necessarily polynomials in $\Z[X]$.

A crucial step is the following lemma.

\begin{lemma}[\.{Z}mija \cite{Zmija}*{Lemma 5}, Heim--Neuhauser \cite{HeimNeu25}*{Lemma 1}]\label{lem:SplittingOfA_n^gModP}
    Let $g$ be a $\Z$-valued arithmetic function such that $g(1) = 1$ and $p$ a prime number. Then
    \[A_{\ell p + r}^g(X) \equiv A_r^g(X) \cdot A_p^g(X)^\ell \mod p\]
    for all $\ell, r \in \N$.

    If $P_p^g(X)$ takes rational integer values at rational integer arguments, we have
    \[A_p^g(X) \equiv X \cdot \left(X^{p-1} - 1\right) \mod p.\]
\end{lemma}

The main argument in the proof to the above lemma is the recursion
\begin{equation*}\label{eq:Ang(S)Recursion}
    A_n^g(X) = X \sum_{k = 1}^n \frac{(n-1)!}{(n-k)!}g(k)A_{n-k}^g(X)
\end{equation*}
for any $n \geq 1$.

\section{Proofs}

In this section, we present the proofs of our results.

\subsection{The form of $A_p^g(X)$ modulo $p$}

In Theorem \ref{thm:Zmija}, \.{Z}mija already emphasized the utility of knowledge regarding the degrees of the irreducible factors of $A_n^g(X)$. An especially nice case occurs if $A_n^g(X)$ is a product of linear factors. In the global setting, i.e. in $\Z[X]$, this happens if and only if the roots of $A_n^g(X)$ are rational. However, if this is not the case, we might achieve this locally modulo a prime number $p$. As an intermediate step, we transform the approach towards the splitting behavior of the D'Arcais polynomials modulo a prime by focusing on the function $g$ instead of $A_n^g(X)$.

\begin{prop}\label{prop:SplittingBehaviorOfA_p^g(X)ModP}
    Let $g$ be a $\Z$-valued arithmetic function such that $g(1) = 1$. Then we have
    \begin{equation*}\label{eq:SplittingBehaviorOfA_p^g(X)ModP}
        A_p^g(X) \equiv X \cdot \left(X^{p-1} - g(p)\right) \mod p
    \end{equation*}
    for any rational prime $p$.
\end{prop}

\begin{proof}
    Define $a_n^g(s)$ via
    \[A_n^g(X) := \sum_{s = 0}^n a_n^g(s)X^{s}\]
    with the convention that $a_n^g(s) = 0$ if $s < 0$ or $s > n$. Following \cite{AGOS}*{Lemma 2.1} and \cite{Stanley}*{Example 5.2.10}, we have
    \begin{align*}
        \sum_{n = 0}^\infty \frac{q^n}{n!} \sum_{s = 0}^n a_n^g(s) X^s &= \exp \left(X \sum_{n  = 1}^\infty g(n) \frac{q^n}{n}\right)\\
        &= \sum_{n = 0}^\infty q^n \sum_{\substack{\lambda \in P(n) \\ \lambda = (1^{m_1}, \ldots, n^{m_n})}} \prod_{j = 1}^n \frac{1}{m_j!} \left(\frac{g(j)X}{j}\right)^{m_j},
    \end{align*}
    yielding
    \begin{equation}\label{eq:CoefficientFormula}
        a_n^g(s) = n!\sum_{\substack{\lambda \in P(n) \\ \lambda = (1^{m_1}, \ldots, n^{m_n})\\
        |\lambda| = s}} \prod_{j = 1}^n \frac{1}{m_j!} \left(\frac{g(j)}{j}\right)^{m_j}.
    \end{equation}
    
    Suppose now that $n = p$ is prime. If $s = 1$, then $\lambda = (1^0, \ldots, (p-1)^0, p^1)$ and hence
    \[a_p^g(1) = p! \frac{g(p)}{p} = (p-1)!g(p).\]
    By Wilson's Theorem, we get $(p-1)! \equiv -1 \mod p$.

    If $s = p$, then $\lambda = (1^p, 2^0, \ldots, n^0)$ and hence
    \[a_p^g(p) = p! \frac{1}{p!} = 1.\]

    If $2 \leq s \leq p - 1$, then $m_j < p$ and $j < p$ for all $\lambda$ in the sum in \eqref{eq:CoefficientFormula}. Hence, the $p$-factor of the factorial in front survives and we get $p \mid a_p^g(s)$.
\end{proof}

By Fermat's Little Theorem, we harvest this low-hanging fruit which is an improvement of \cite{Zmija}*{Proposition 12} and similar previous results.

\begin{cor}\label{cor:SplittingIfg(p)IsNice}
    For any prime number $p$, the polynomial $A_p^g(X)$ modulo $p$ splits into linear factors if and only if $g(p) \equiv 0,1 \mod p$.
\end{cor}

\begin{rmk}
    \.{Z}mija \cite{Zmija}*{Lemma 5} showed that, if $P_p^g(X)$ takes rational integer values at rational integer arguments, then
    \[A_p^g(X) \equiv X \cdot (X - 1) \dots (X - p + 1) \mod p\]
    which corresponds to our case $g(p) \equiv 1 \mod p$. 
    
    For a contrarious example, consider $g$ being the identity $\mathrm{id}: \N \longrightarrow \Z$, satisfying $\mathrm{id}(p) = p \equiv 0 \mod p$ for all rational primes $p$. Then $P_n^{\mathrm{id}}(X)$ is related to the Laguerre polynomials and in particular, we have
    \[P_2^{\mathrm{id}}(1) = \frac{1}{2} \left(\mathrm{id}(1) P_1^{\mathrm{id}}(1) + \mathrm{id}(2) P_0^{\mathrm{id}}(1)\right) = \frac{1}{2}\left(1 \cdot \frac{1}{1} \cdot 1 \cdot 1 + 2 \cdot 1\right) = \frac{3}{2}\]
    which is not a rational integer.
\end{rmk}

\subsection{Integers on parallel lines of the real axis}

By virtue of Lemma \ref{lem:IndexIsDetOfMatrix}, the index $\kappa_\alpha$ is given by the absolute value of the determinant of a matrix $M$, sending an integral basis of $\Q(\alpha)$ to the power basis generated by $\alpha$. For the two most commonly treated kinds of number fields, we explicitly apply this in the following lemma.

\begin{lemma}\label{lem:IndexOfaZeta_mAndaOmega_D}
    Let $a, b \in \Z$.
    \begin{enumerate}
        \item If $m \geq 3$, then
        \begin{equation*}
            \kappa_{a\zeta_m + b} = |a|^{\frac{\varphi(m)\left(\varphi(m) - 1\right)}{2}}
        \end{equation*}
        where $\varphi$ is the Euler totient function.
        \item If $D \in \Z$ square-free, then
        \begin{equation*}
            \kappa_{a\omega_D + b} = |a|. 
        \end{equation*}
    \end{enumerate}
\end{lemma}

\begin{proof}
    We only prove the first item as the second follows in a similar manner. Because of $\deg [\Q(\zeta_m) : \Q] = \varphi(m)$, we have
    \[\Z[a\zeta_m + b] \cong \prod_{j = 0}^{\varphi(m) - 1} (a\zeta_m + b)^j\Z.\]
    Since $\mathcal{O}_{\Q(\zeta_m)} = \Z[\zeta_m]$, we may apply Lemma \ref{lem:IndexIsDetOfMatrix} which yields the claim by Gau\ss's Little Theorem since the corresponding matrix $M$ is a lower triangular matrix with $1, a, \ldots, a^{\varphi(m) - 1}$ as diagonal entries.
\end{proof}

\begin{rmk}
    Neuhauser and the first author \cite{HeimNeu25}*{Lemma 3} showed that, if $\mu \in \Z$ with $p \mid \mu$ for a rational prime $p$, then
    \[p \nmid \kappa_{\zeta_m + \mu\alpha}\]
    for all $\alpha \in \Z[\zeta_m]$. Their proof uses more general algebraic tools but may be translated into this setting.
\end{rmk}

With this in hand, we move on to prove the first of our main results which extends the range of non-roots along parallel lines of the real axis.

\begin{proof}[Proof of Theorem \ref{thm:NonRootsTranslatingWithRationalIntegers}]
    We only prove the second item as the others follow in a similar manner. Due to Lemma \ref{lem:IndexOfaZeta_mAndaOmega_D}, we may apply the Dedekind--Kummer Theorem to $\alpha = a\zeta_m + b$. Since we have $3 \not \equiv 1 \mod m_3$, the minimal polynomial $m_\alpha$ does not decompose into linear factors modulo $3$ and is therefore divisible by a non-linear factor. Invoking Corollary \ref{cor:SplittingIfg(p)IsNice} yields the claim since $A_n^g(X)$ factorizes into linear polynomials and hence $m_\alpha(X) \nmid P_n^g(X)$ locally as well as globally.
\end{proof}

For our second main result, we point out that
\begin{equation}\label{eq:A_1^gA_2^g}
    A_1^g(X) = X
\end{equation}
regardless of the arithmetic function $g$.

\begin{proof}[Proof of Theorem \ref{thm:NotRamifiedThm}]
    Due to $p \nmid a$ and Lemma \ref{lem:IndexOfaZeta_mAndaOmega_D}, we may employ the Dedekind--Kummer Theorem again. Following Lemma \ref{lem:SplittingOfA_n^gModP} and Proposition \ref{prop:SplittingBehaviorOfA_p^g(X)ModP}, we have
    \[A_n^g(X) \equiv X^n \mod p\]
    because $g(p) \equiv 0 \mod p$ and $n \equiv 0, 1 \mod p$, using \eqref{eq:A_1^gA_2^g} in the latter case if necessary. Since $p \nmid D$, we deduce that $p$ does not ramify in $\Q(\omega_D)$, hence $m_{a\omega_D + b}(X) \nmid A_n^g(X)$ modulo $p$ and the first item follows. 

    For the second claim, we deduce that $p$ does not split in $\Q(\omega_D)$ and hence, $m_{a\omega_D + b}(X)$ is an irreducible quadratic polynomial modulo $p$. However, by the same results as above, we have
    \[A_n^g(X) \equiv \left(X\left(X^{p-1} - 1\right)\right)^{\floor{\frac{n}{p}}}X^t \mod p\]
    where $t \in \{0,1\}$ with $n \equiv t \mod p$. Clearly, $m_{a\omega_D + b}(X) \nmid A_n^g(X)$ as desired.
\end{proof}

The above proofs indicate that---not too surprisingly---more knowledge about $g$ pushes our ability to exclude further potential roots. Our main interest still lies with the original case of the Dedekind eta function and the $P_n^\sigma(X)$. Since $\sigma(p) \equiv 1 \mod p$ for every rational prime $p$, it suffices to compute each polynomial up to $A_p^\sigma(X)$ to know the splitting behaviors for each $A_n^\sigma(X)$ modulo $p$. The first few $P_n^\sigma(X)$, resp. $A_n^\sigma(X)$, look like this
\begin{align*}
    A_0^\sigma(X) = 0!P_0^\sigma(X) &= 1,\\
    A_1^\sigma(X) = 1!P_1^\sigma(X) &= X,\\
    A_2^\sigma(X) = 2!P_2^\sigma(X) &= X (X + 3),\\
    A_3^\sigma(X) = 3!P_3^\sigma(X) &= X(X+8)(X+1),\\
    A_4^\sigma(X) = 4!P_4^\sigma(X) &= X (X + 14)(X + 3)(X + 1),\\
    A_5^\sigma(X) = 5!P_5^\sigma(X) &= X (X + 6)(X + 3)(X^2 + 21X + 8),\\
    A_6^\sigma(X) = 6!P_6^\sigma(X) &= X (X + 10)(X + 1)(X^3 + 34X^2 + 181X + 144),
\end{align*}
where each factor is irreducible. Investigating these products further modulo $7$ reveals more non-roots along the imaginary axis.

\begin{proof}[Proof of Theorem \ref{thm:NonRootsOnImaginaryAxisWithSigma}]
    The minimal polynomial of $ai + b$ is given by $m_{ai + b}(X) = (X - b)^2 + a^2$. Suppose $3 \nmid a$, then $P_n^\sigma(ai + b) \neq 0$ by virtue of the second item of Theorem \ref{thm:NonRootsTranslatingWithRationalIntegers} for all $n \in \N$. 
     
    Now assume $3 \mid a$. By our assumption on $a$, it follows $7 \nmid a$ and thus by Lemma \ref{lem:IndexOfaZeta_mAndaOmega_D}, we have $7 \nmid \kappa_{ai + b}$. Then $m_{ai + b}(X)$ does not decompose into irreducible linear polynomials modulo $7$ by the Dedekind--Kummer Theorem since $7 \equiv 3 \mod 4$ but instead remains a quadratic polynomial.
    \begin{enumerate}
        \item If $n \not \equiv 5, 6 \mod 7$, then $A_n^\sigma(X)$ decomposes into irreducible linear polynomials over $\Z$ and hence does so modulo $7$. Thus, $m_{ai + b}(X) \nmid P_n^\sigma(X)$.

        If $n \equiv 6 \mod 7$, then the only non-linear irreducible polynomial in the decomposition of $A_n^\sigma(X)$ modulo 7 is
        \[h_6(X) \equiv X^3 - X^2 - X + 4 \mod 7.\]
        Since $m_{ai +b}(X)$ and $h_6(X)$ are both irreducible, we again deduce $m_{ai+b}(X) \nmid P_n^\sigma(X)$.
        \item If $n \equiv 5 \mod 7$, the only non-linear irreducible polynomial in the decomposition of $A_n^\sigma(X)$ modulo 7 is
        \[h_5(X) \equiv X^2 + 1 \mod 7.\]
        If $7 \nmid b$, then $h_5(X)$ is obviously not associated to $m_{ai + b}(X)$ modulo 7 due to the latter having a non-vanishing linear term. If however $7 \mid b$, then
        \[m_{ai + b}(X) \equiv X^2 + a^2 \mod 7\]
        which is congruent to $h_5(X)$ if and only if $a^2 \equiv 1 \mod 7$.
    \end{enumerate}
    This completes the proof.
\end{proof}

\section{Open challenges}

As pointed out in the last section of \cite{HeimNeu25}, the search for further non-roots of the D'Arcais polynomials may advance in various directions. As this paper was meant to focus on methods to manipulate the application of the Dedekind--Kummer Theorem, entirely different approaches may provide even deeper insights to the structure of the D'Arcais polynomials.

One of them is the investigation of the behavior of $P_n^g(X)$ along certain stretched copies of the unit circle. Let $a \in \R_{\geq 0}$ and
\[\mathbb{U}_a := \{z \in \C : |z| = a\}.\]
What can we say about (non-)roots of $P_n^g(X)$ in $\mathbb{U}_a$?

Since the constant coefficient of $P_n^g(X)$ always vanishes, one may examine $H_n^g(X) := \frac{P_n^g(X)}{X}$. In the case of $g = \sigma$, one finds that $H_n^\sigma(X)$ only has non-negative coefficients. Hence, $H_n^\sigma(x) > 0$ for all non-negative real $x$. A polynomial $f(X)$ is called \emph{Hurwitz} if $f(x) = 0$ implies $\Re (x) < 0$, as was introduced in \cite{Hurw}. Neuhauser and the first author conjectured that the $H_n^\sigma(X)$ are Hurwitz. One might want to apply an algorithm suggested by Holtz \cite{Holtz}.

\section*{Acknowledgements}

We thank Badri Vishal Pandey for valuable suggestions.

\end{document}